\documentclass[a4paper,12pt,reqno,twoside]{amsart}

\usepackage[utf8]{inputenc} 		
\usepackage[T1]{fontenc} 

\usepackage{amssymb}
\usepackage{amsmath}
\usepackage{graphicx}
\usepackage[colorlinks,citecolor=blue,urlcolor=blue]{hyperref}
\usepackage{cite}
\usepackage[hmargin=2.5cm,vmargin=2.5cm]{geometry}
\usepackage{mathrsfs} 
\usepackage{enumerate}

\newcommand{\be}{\begin{eqnarray}}
\newcommand{\ee}{\end{eqnarray}}

\newcommand{\beq}{\begin{equation}}
\newcommand{\eeq}{\end{equation}}
\newcommand{\beqn}{\begin{equation*}}
\newcommand{\eeqn}{\end{equation*}}

\newcommand{\slot}{\,\cdot\,}
\newcommand{\round}[1]{\lfloor#1\rfloor}

\newtheorem{thm}{Theorem}[section]

\newtheorem{cor}[thm]{Corollary}
\newtheorem{lem}[thm]{Lemma}
\newtheorem{defn}[thm]{Definition}

\newtheorem{remark}[thm]{Remark}

\newcommand\cC{{\mathcal C}}

\newcommand\cL{{\mathcal L}}
\newcommand\cM{{\mathcal M}}

\newcommand\cT{{\mathcal T}}

\newcommand\bN{{\mathbb N}}

\newcommand\bR{{\mathbb R}}

\newcommand{\ve}{\varepsilon}

\def\bfT{\mathbf{T}}

\begin{document}

\title{Quasistatic dynamics with intermittency}

\author[Juho Lepp\"anen]{Juho Lepp\"anen}
\address[Juho Lepp\"anen]{
Department of Mathematics and Statistics, P.O.\ Box 68, Fin-00014 University of Helsinki, Finland.}
\email{juho.leppanen@helsinki.fi}

\author[Mikko Stenlund]{Mikko Stenlund}
\address[Mikko Stenlund]{
Department of Mathematics and Statistics, P.O.\ Box 68, Fin-00014 University of Helsinki, Finland.}
\email{mikko.stenlund@helsinki.fi}
\urladdr{http://www.math.helsinki.fi/mathphys/mikko.html}

\keywords{Quasistatic dynamical system, intermittency, Pomeau--Manneville map, ergodic theorem, physical family of measures}

\thanks{2010 {\it Mathematics Subject Classification.} 37C60; 37D25, 37A10, 37A30} 
 


\begin{abstract}
We study an intermittent quasistatic dynamical system composed of nonuniformly hyperbolic Pomeau--Manneville maps with time-dependent parameters. We prove an ergodic theorem which shows almost sure convergence of time averages in a certain parameter range, and identify the unique physical family of measures. The theorem also shows convergence in probability in a larger parameter range. In the process, we establish other results that will be useful for further analysis of the statistical properties of the model.
\end{abstract}

\maketitle


\subsection*{Acknowledgements}
This work was supported by the Jane and Aatos Erkko Foundation, and by Emil Aaltosen S\"a\"ati\"o. 


\section{Introduction}
In recent years time-dependent dynamical systems have gained an increasing amount of attention in the mathematics --- especially mathematical physics --- literature; see, for instance,~\cite{LasotaYorke_1996,ConzeRaugi_2007,OttStenlundYoung_2009,Stenlund_2011,Nandori_2012,StenlundYoungZhang_2013,GuptaOttTorok_2013,Kawan_2014,Kawan_2015,Aimino_etal_2015,DobbsStenlund_2014}. On the other hand, nonuniformly hyperbolic dynamical systems with neutral fixed points have been studied extensively at least since the Pomeau--Manneville map was proposed as a model of the intermittent behavior of turbulent fluids~\cite{PomeauManneville_1980, GaspardWang_1988}. In this paper we introduce a class of intermittent quasistatic dynamical systems, which combines the two, and initiate the study of their statistical properties. We begin the section by introducing the constituent intermittent dynamics and then proceed to defining their quasistatic conglomerate system.

\medskip
\noindent{\bf Notations and conventions.} The sigma-algebra on a topological space will always be that of the Borel sets. We denote the Lebesgue measure on~$[0,1]$ by~$m$ and write
$
m(f) = \int f\,dm = \int f\,dx
$
for the Lebesgue integral of a function $f:[0,1]\to\bR$, as well as $\int_a^b f\,dx = m(f1_{[a,b]})$ for $0\le a\le b\le 1$. The~$L^p$ spaces are defined in terms of the Lebesgue measure. In particular, $\|f\|_1 = \int |f|\,dx$. The space of continuous functions from~$[0,1]$ to~$\bR$, denoted by $C([0,1])$, is equipped with the uniform norm and the resulting Borel sigma-algebra.

We use the usual floor and ceiling functions for rounding real numbers~$x$ to the nearest integer:
\beqn
\lfloor x \rfloor = \max\{n\in \bN\,:\,n\le x\} 
\quad\text{and}\quad
\lceil x \rceil = \min\{n\in \bN\,:\,n\ge x\} .
\eeqn

System constants, whose values are determined solely by the parameters defining the model, are distinguished by a subindex (e.g., $C_0$) from generic constants (e.g., $C$), whose values may change from one expression to the next.

\subsection{A class of intermittent maps}

As in~\cite{LiveraniSaussolVaienti_1999}, let us define, for each $\alpha\in(0,1)$, the map $T_\alpha:[0,1]\to[0,1]:$
\beq\label{eq:PM}
T_\alpha(x) = 
\begin{cases}
x(1+2^{\alpha}x^{\alpha}) & x\in [0,\tfrac12)
\\
2x-1 & x\in [\tfrac12,1].
\end{cases}
\eeq

It follows from~\cite{LiveraniSaussolVaienti_1999} that $T_\alpha$ has an invariant SRB (Sinai--Ruelle--Bowen) measure $\hat\mu_\alpha$. This probability measure is equivalent to the Lebesgue measure~$m$. We denote the density by~$\hat h_\alpha$.

The word ``intermittent'' derives from the dynamical characteristics of such maps: The dynamics is strongly chaotic due to expansion except near the neutral (or indifferent) fixed point at the origin, where the derivative equals $1$. Once the orbit lands near the origin, it stays in its small neighborhood for a very long time before expansion has a noticeable effect again. The parameter~$\alpha$ dictates how neutral the fixed point is; the larger~$\alpha$ the longer it takes for the orbit to return to the expanding region of the interval. Consequently, correlation functions decay at a slower and slower rate (roughly as $n^{-(\frac1\alpha-1)}$, where~$n$ is the time) and the invariant density becomes more and more concentrated around the origin as~$\alpha$ increases (with an estimate $\hat h_\alpha(x)\lesssim x^{-\alpha}$). In the parameter range $0<\alpha<1$ studied here, the invariant density is nevertheless integrable. Pomeau--Manneville-type maps were originally proposed as models of the intermittent behavior observed in studies of turbulent fluids; see~\cite{PomeauManneville_1980, GaspardWang_1988}. Since then, such maps have been investigated to the extent that we do not even attempt to cover the liteature here, besides the references cited below.

\subsection{Quasistatic dynamical systems}

The notion of an abstract quasistatic dynamical system, which we now recall, was introduced in~\cite{DobbsStenlund_2014}. Discussions on its physical interpretation and significance can be found in~\cite{DobbsStenlund_2014,Stenlund_2015}.

\begin{defn}
Let~$(X,\mathscr{F})$ be a measurable space, $\cM$~a topological space whose elements are measurable self-maps $T:X\to X$, and~$\bfT$ a triangular array of the form
\beqn\label{eq:array}
\bfT = \{T_{n,k}\in\cM\ :\ 0\le k\le n, \ n\ge 1\} .
\eeqn
If there exists a piecewise continuous curve $\Gamma:[0,1]\to\cM$ such that
\beqn
\lim_{n\to\infty}T_{n,\round{nt}} = \Gamma_t
\eeqn
we say that $(\bfT, \Gamma)$ is a \emph{quasistatic dynamical system (QDS)} with \emph{state space}~$X$ and \emph{system space}~$\cM$. 
\end{defn}

Before describing the dynamics, let us define the intermittent QDS promised. 

\begin{defn}[An intermittent QDS]
Let~$X = [0,1]$ and $\cM = \{T_\alpha\,:\,0\le\alpha\le 1\}$ (equipped, say, with the uniform topology). Next, let
\beqn
\{\alpha_{n,k}\in [0,1]\ :\ 0\le k\le n, \ n\ge 1\}
\eeqn
be a triangular array of parameters and
\beqn
\gamma : [0,1] \to [0,1]
\eeqn
a piecewise continuous curve satisfying
\beqn
\lim_{n\to\infty} \alpha_{n,\round{nt}} = \gamma_t.
\eeqn
Finally, define $\Gamma_t = T_{\gamma_t}$ and
\beqn\label{eq:T_nk}
T_{n,k} = T_{\alpha_{n,k}}.
\eeqn
\end{defn}
It will often be convenient to use the parameter representation instead of the system-space representation, i.e., to refer to $\alpha_{n,k}$ and $\gamma$ instead of $T_{n,k}$ and $\Gamma$.

The time evolution of an initial state~$x\in [0,1]$ is now given by the triangular array~$\bfT$, separately on each level of the array. That is, given~$n\ge 1$, the point 
\beqn
x_{n,k} = T_{n,k}\circ\dots\circ T_{n,1}(x) \in [0,1]
\eeqn
is the state after~$k\in\{0,\dots,n\}$ steps on the~$n$th level of~$\bfT$. (We define $x_{n,0} = x$.) 

Introducing the continuous parameter $t\in[0,1]$ and setting $k = \round{nt}$, the piecewise constant curve~$t\mapsto \alpha_{n,\round{nt}}$ approximates~$\gamma$ with ever-increasing accuracy, as $n\to\infty$. In the physics terminology, it is helpful to think of~$t$ as macroscopic time and, with $n$ fixed, of $k = \round{nt}$ as the corresponding microscopic time.

Given a measurable function $f:[0,1]\to\bR$, we denote
\beqn
f_{n,k} = f\circ T_{n,k}\circ\dots\circ T_{n,1}, \quad 0\le k \le n.
\eeqn
(We define $f_{n,0} = f$.) We define the functions $S_n:[0,1]\times [0,1]\to\bR$ by
\beqn
S_n(x,t) = \int_0^{nt} f_{n,\round{s}}(x)\, ds, \quad n\ge 1.
\eeqn
Note that, given~$x$, the function~$S_n(x,\slot)$ is a piecewise linear interpolation of the Birkhoff-type sums~$\sum_{k=0}^{\round{nt}-1}  f_{n,k}(x)$, $t\in[0,1]$, and, as such, an element of $C([0,1])$. Next, we define
\beqn\label{eq:zeta}
\zeta_n(x,t) = n^{-1} S_n(x,t) = \int_0^{t} f_{n,\round{ns}}(x)\,d s.
\eeqn
Given an initial distribution~$\mu$ for~$x$, we can view~$\zeta_n$ as a random element of~$C([0,1])$.

In this paper we are interested in identifying conditions under which
\beqn\label{eq:as}
\lim_{n\to\infty}\zeta_n = \zeta
\eeqn
almost surely with respect to an initial measure, where $\zeta\in C([0,1])$ has the expression
\beqn\label{eq:zeta_limit}
\zeta(t) = \int_0^t \hat\mu_{\gamma_s}(f)\,d s.
\eeqn
For a heuristic argument why such a result might be true, see~\cite{Stenlund_2015}.
In the fully degenerate case, in which $T_{n,k}=T_\alpha$ for some $\alpha$ and all~$k$ and~$n$, Birkhoff's ergodic theorem guarantees that, given an $L^1$ function~$f$, $\lim_{n\to\infty}\zeta_n(x,t) = t\int f \,d\hat\mu_\alpha$ for almost every~$x$ with respect to the SRB measure~$\hat\mu_\alpha$ (equivalently~$m$), for all~$t\in[0,1]$. Due to the lack of an invariant measure, our approach in the proper QDS setup is not based on Birkhoff's ergodic theorem.

The notion of a physical family of measures, proposed in~\cite{Stenlund_2015}, is closely related to the problem above. Specializing to the setting of the current paper, let us recall the definition.

\begin{defn}\label{defn:physical_family}
Let $\mathscr{P} = (\mu_t)_{t\in[0,1]}$ be a one-parameter family of measures on~$[0,1]$, and suppose that the map $t\mapsto \mu_s(f)$ is measurable for all bounded continuous functions $f:[0,1]\to\bR$. Suppose there exists a measurable set~$A\subset [0,1]$ with $m(A)>0$ such that 
\beqn
\lim_{n\to\infty} \int_0^{t} f_{n,\round{ns}}(x)\,d s = \int_0^t \mu_s(f)\,d s, \quad t\in[0,1],
\eeqn
holds for all $x\in A$ and all bounded continuous functions~$f:[0,1]\to\bR$. Then we say that~$\mathscr{P}$ is a \emph{physical family of measures} for the QDS. The set~$A$ is called a \emph{basin} of~$\mathscr{P}$. 
\end{defn}
Thus, a physical family of measures captures, in the limit~$n\to\infty$, the mean behavior of the observations $f_{n,k}(x)$, $k\ge 0$, for a notable fraction of initial points~$x$ \emph{with respect to the Lebesgue measure~$m$}, which we point out is not invariant in any sense.

\subsection{Main theorem}

The following theorem is the main result of our paper, which we choose to present here because it can now be understood without additional preparation. In the course of its proof, we establish other results, interesting in their own right, that will be beneficial to further research on the statistical properties of the intermittent QDS introduced above (or similar models).

\begin{thm}\label{thm:main}
Suppose that the curve~$\gamma:[0,1]\to[0,1]$ is piecewise H\"older continuous\footnote{Precisely, there exists a finite partition of $[0,1]$ into disjoint intervals $I_1,\dots,I_m$ such that $\gamma$ is H\"older continuous on each $I_i$ with exponent~$\theta$.} with exponent $ \theta\in(0,1]$, that 
\beq\label{eq:avoid_boundaries}
\overline{\gamma([0,1])} \subset [0,\beta_*]
\eeq
for some $\beta_*\in(0,1)$, and that
\beq\label{eq:convergence_to_curve}
\lim_{n\to\infty} n^{\theta}\sup_{t\in[0,1]}|\alpha_{n,\round{nt}} - \gamma_t| < \infty.
\eeq
\begin{enumerate}[(i)]
\item If $\beta_*\ge \frac12$, then for each $f\in C([0,1])$,
\beq\label{eq:limit}
\lim_{n\to\infty}\sup_{t\in[0,1]}|\zeta_n(x,t) - \zeta(t)| = 0
\eeq
in probability, with respect to the Lebesgue measure. That is,
\beqn
\lim_{n\to\infty}m\!\left(\sup_{t\in[0,1]}|\zeta_n(x,t) - \zeta(t)|\ge \ve\right) = 0
\eeqn
for all $\ve>0$.

\medskip
\item If $\beta_*<\frac12$, then \eqref{eq:limit} holds for almost every~$x\in[0,1]$ with respect to the Lebesgue measure. The one-parameter family of measures $\mathscr{P} = (\hat\mu_{\gamma_t})_{t\in[0,1]}$ is the unique physical family of measures, and it has a basin $A$ of full measure, i.e., $m(A)=1$.
\end{enumerate}
\end{thm}

The condition in~\eqref{eq:avoid_boundaries} is needed to maintain uniform control in certain estimates, mainly on the polynomial correlation decay, while the convergence rate in~\eqref{eq:convergence_to_curve} is the natural one suggested by the ``equipartition''~$\alpha_{n,k} = \gamma_{kn^{-1}}$. The assumption that the curve~$\gamma$ is H\"older continuous has a physical motivation: it could, for instance, be a Brownian path, as a consequence of external forcing on the system, and allowing jumps to occur accommodates for rare, sudden, changes in the forcing. Regarding the uniqueness of the physical family of measures, we recall that two such families, $(\mu_{\gamma_t})_{t\in[0,1]}$ and $(\mu_{\gamma_t}')_{t\in[0,1]}$, are considered the same if $\mu_{\gamma_t} = \mu_{\gamma_t}'$ except for a zero-measure set of parameters~$t$.

We reiterate that the rate of correlation decay for the intermittent QDS is polynomial. The essential distinction between cases~(i) and~(ii) is the summability of this rate in case~(ii), which yields stronger results.

As a matter of fact, we prove a bit more regarding case (i). Namely, if $f\in C^1([0,1])$, there exists a subsequence $(n_k)_{k\ge 1}$ along which the convergence in \eqref{eq:limit} is almost sure \emph{and} $\lim_{k\to\infty}n_{k+1}/n_k = 1$. Thus, the sequence converges almost surely, save for gaps of sublinearly growing length. One would like to patch the gaps and promote the result to almost sure convergence of the entire sequence, say, with the aid of the Borel--Cantelli lemma. We have been unable to do so. Alternatively, one could directly try to appeal to a sufficiently general strong law of large numbers for triangular arrays of random variables, but we have failed to find a suitable one in the probability theory literature. Either way, the critical difficulties encountered in establishing almost sure convergence in case~(i) have to do with the necessity of controlling
\beqn
\zeta_n(t)-\zeta_m(t) = \int_0^t f\circ T_{n,\round{ns}}\circ\dots\circ T_{n,1} - f\circ T_{m,\round{ms}}\circ\dots\circ T_{m,1} \,ds
\eeqn
with $n\ne m$, which involves comparing different levels of the array $\bfT$, consisting of different maps.

In view of \eqref{eq:avoid_boundaries}, let us also point out that the QDS changes its nature if the curve is allowed to meet the right-hand boundary of $[0,1]$. For instance, the invariant density of~$T_1$ is non-integrable. Nevertheless, it is possible that the result remains true, even in the almost sure sense, if the curve makes some isolated visits to the boundary. As we have tried to convey above, investigating this admittedly interesting question would require a notable amount of technical innovation.

Similar results for certain uniformly hyperbolic QDSs, enjoying exponential correlation decay, were obtained in~\cite{Stenlund_2015}. The most significant difference here is that the intermittent QDS is nonuniformly hyperbolic. Nonuniformly hyperbolic systems are inherently more difficult to analyze due to the intermittent character of the dynamics alluded to above. To make matters worse, the system at issue --- as a QDS --- is time dependent. The study of time-dependent Pomeau--Manneville maps was recently initiated in~\cite{Aimino_etal_2015}, where a statistical memory-loss result for compositions of such maps was established; that result will have a role in the proof of Theorem~\ref{thm:main}. 

Finally, the polynomial rate of correlation decay discussed above will also have an impact on other kinds of limit theorems in the future, including the assortment of techniques suitable for their proofs. For instance, one expects a central limit theorem to hold (only) in case (ii).

\medskip
\noindent{\bf How the paper is organized.}
In Section~\ref{sec:preliminaries} we discuss technical facts needed for understanding the subsequent sections. In Section~\ref{sec:C^1_cone} we elaborate on the connection of two classes of functions appearing in the proofs; this section also facilitates bounding multiple correlation functions, which is carried out in Section~\ref{sec:decorrelation}. In Section~\ref{sec:perturbation} we study the parameter dependence of the transfer operator and of the SRB density. The proof of the main result, Theorem~\ref{thm:main}, is finally given in Section~\ref{sec:proof_main}.




\medskip
\section{Preliminaries}\label{sec:preliminaries}

We denote the transfer operator associated to $T_\alpha$ and the Lebesgue measure by $\cL_\alpha$. That is,
\beqn
\cL_\alpha f(x) = \sum_{y\in T_\alpha^{-1}x}\frac{f(y)}{T_\alpha'(y)} = \frac{f(y_\alpha)}{T_\alpha'(y_\alpha)} + \frac{f(\frac x2+\frac12)}{2},
\eeqn
where $y_\alpha = y_\alpha(x)$ is the preimage of $x$ under the left branch $T_\alpha|_{[0,\frac12)}$. The following result establishes an important invariance property for the transfer operator:

\begin{lem}\label{lem:cone}
Let $0<\beta<1$ and $a(\beta) = 2^{\beta}(\beta +2)$. The transfer operator~$\cL_\alpha$ maps the convex cone
\beqn
\begin{split}
\cC_*(\beta) = \{f\in C((0,1])\cap L^1\,:\, & \text{$f\ge 0$, $f$ decreasing,} 
\\
& \text{$x^{\beta+1}f$ increasing, $f(x)\le a(\beta) x^{-\beta} m(f)$}\}
\end{split}
\eeqn
into itself, provided $0<\alpha\le \beta$.
\end{lem}

For a proof of Lemma~\ref{lem:cone}, see~\cite{LiveraniSaussolVaienti_1999} for the original case of a single parameter ($\alpha = \beta$) and~\cite{Aimino_etal_2015} for the above adaptation to a range of parameters ($0<\alpha\le \beta$).

It follows from~\cite{LiveraniSaussolVaienti_1999} that the SRB density of $T_\alpha$ satisfies
\beqn
\hat h_\alpha \in \cC_*(\alpha).
\eeqn
Note that the cone is increasing, i.e., $\cC_*(\alpha)\subset\cC_*(\beta)$ if $\alpha\le\beta$. In particular, $\hat h_\alpha\in\cC_*(\beta)$ whenever~$\alpha\le\beta$.

We will always consider $\beta_*\in(0,1)$ fixed, and write
\beqn
a = a(\beta_*)
\quad\text{and}\quad
\cC_* = \cC_*(\beta_*)
\eeqn
for brevity. We will call a sequence $(T_{\alpha_i})_{i\ge 1}$ of Pomeau--Manneville maps an \emph{admissible sequence} if $\alpha_i\le \beta_*$ for all~$i$. To keep notations simple, we write
\beqn
T_i = T_{\alpha_i}
\quad\text{and}\quad
\cL_i = \cL_{\alpha_i}
\eeqn
for such a sequence. We also denote
\beqn
\widetilde T_{n,m} = T_n\circ\dots\circ T_m
\quad\text{and}\quad
\widetilde\cL_{n,m} = \cL_n\cdots\cL_m,\quad m<n,
\eeqn
together with
\beqn
\widetilde T_n = T_n\circ\dots\circ T_1
\quad\text{and}\quad
\widetilde\cL_n = \cL_n\cdots\cL_1.
\eeqn
Nearly all the results below apply for general admissible sequences, without additional restrictions on $\beta_*$.
Only in Section~\ref{sec:proof_main} do we need to assume $\beta_*<\frac12$ in order to check Condition (A) of Lemma~\ref{lem:QS_mean}, which requires a summable rate of correlation decay. 

Let us already recall the following key estimate from~\cite{Aimino_etal_2015}; see also ~\cite{LiveraniSaussolVaienti_1999} for a similar result in the case of a single map instead of a sequence.
\begin{lem}\label{lem:Aimino}
There exists a constant $C_0 = C_0(\beta_*)>0$ such that the following holds.
Let~$(T_i)_{i\ge 1}$ be admissible and $f,g\in\cC_*$ with $\int f\,dx = \int g\,dx$. Then, for all $n\ge 0$,
\beqn
\|\widetilde\cL_n (f-g)\|_1 \le C_0(\|f\|_1+\|g\|_1) \rho(n) 
\eeqn
where
\beqn
\rho(n) = n^{-(\frac1{\beta_*}-1)}(\log n)^\frac1{\beta_*}, \quad n\ge 2,
\eeqn
and $\rho(0) = \rho(1) = 1$.
\end{lem}


\medskip
\section{$C^1$ functions and the cone $\cC_*$}\label{sec:C^1_cone}
Because of the invariance property (Lemma~\ref{lem:cone}), much of the initial technical work is carried out for functions belonging to the cone~$\cC_*$. However, in applications a more familiar class of functions is preferred. The following lemma --- which yields a recipe for passing results from $\cC_*$ to $C^1$ --- is (essentially) from~\cite{LiveraniSaussolVaienti_1999}. Since a detailed proof seems not to have been published, we provide it below as community service. 
\begin{lem}\label{lem:C^1_cone}
There exists a constant $C_1 = C_1(\beta_*)>0$ such that the following holds. Suppose~$A,B\ge 0$. There exist numbers $\lambda<0$, $\nu>0$ and $\delta>0$ such that
\beqn
(f + \lambda x + \nu)h + \delta \in \cC_*
\eeqn
with
\beqn
\|(f + \lambda x + \nu)h + \delta\|_1 \le C_1AB
\eeqn
for every $f\in C^1([0,1])$ with $\|f\|_{C^1}\le A$ and every $h\in\cC_*$ with $m(h)\le B$.
In particular,
$
(\lambda x + \nu)h + \delta \in \cC_*
$
with
$
\|(\lambda x + \nu)h + \delta\|_1 \le C_1AB. 
$
\end{lem}

\begin{remark}
The proof shows that there exist system constants $\lambda_1<0$, $\nu_1>0$ and~$\delta_1>0$ such that
\beqn
\lambda = A\lambda_1, \quad \nu = A\nu_1, \quad \text{and}\quad \delta = AB\delta_1.
\eeqn
While it is true that a similar result can be obtained with $\lambda = \|f\|_{C_1}\lambda_1$, $\nu = \|f\|_{C_1}\nu_1$ and~$\delta = \|f\|_{C_1} m(h)\delta_1$ without imposing bounds on $\|f\|_{C_1}$ and $m(h)$, there is some virtue in having the constants depend on the upper bound only. For example, in the current formulation the function $(\lambda x + \nu)h + \delta$ is automatically in the cone.
\end{remark}

\begin{proof}[Proof of Lemma~\ref{lem:C^1_cone}]
First note that fixing
\beqn
\lambda \le -\|f'\|_\infty
\eeqn
implies that $f+\lambda x$ is decreasing. The core idea of the proof is now that, for large $\nu>0$, the function
\beqn
\left(\frac{f + \lambda x}\nu + 1\right)\! h
\eeqn
is very close to $h\in\cC_*$, and thus belongs to the cone after a small modification. To implement this idea rigorously, we write
\beqn
(f + \lambda x + \nu)h + \delta = \nu\left(\left(\frac{f + \lambda x}\nu + 1\right)\! h + \frac\delta\nu\right) = \nu(gh + \delta'),
\eeqn
where 
\beq\label{eq:g_and_f}
g = \frac{f + \lambda x}\nu + 1 \quad\text{and}\quad \delta' = \frac\delta\nu.
\eeq
It now suffices to show that if $g\in C^1([0,1])$ is an arbitrary function such that $\|g-1\|_{C^1}$ is sufficiently small and $g'\le 0$, there exists $\delta'>0$ such that
\beqn
\psi = gh + \delta' \in \cC_*.
\eeqn

\noindent{\bf Step 1.} Suppose $g\ge 0$ and $g'\le 0$. Then $\psi\ge 0$ and~$\psi$ is decreasing.

\noindent{\bf Step 2.} We identify a condition which guarantees that $x^{\alpha+1}\psi$ is increasing. To that end, let $x<y$ and observe that
\beqn
\begin{split}
& y^{\alpha+1}\psi(y) - x^{\alpha+1}\psi(x) 
\\
& = g(y) y^{\alpha+1}h(y) - g(x) x^{\alpha+1}h(x)  + \delta'( y^{\alpha+1} -  x^{\alpha+1})
\\
& = g(y) [y^{\alpha+1}h(y)-x^{\alpha+1}h(x)] - (g(x)-g(y))x^{\alpha+1}h(x) + \delta'( y^{\alpha+1} -  x^{\alpha+1})
\\
& \ge -(g(x)-g(y))x^{\alpha+1}h(x) + \delta'( y^{\alpha+1} -  x^{\alpha+1}),
\end{split}
\eeqn
where we used that $x^{\alpha+1}h$ is increasing and $g\ge 0$. Since $g$ is a decreasing~$C^1$ function and $0\le h(x)\le ax^{-\alpha}m(h)$,
\beqn
-(g(x)-g(y))x^{\alpha+1}h(x) \ge -\|g'\|_\infty(y-x) ax\,m(h).
\eeqn
On the other hand,
\beqn
y^{\alpha+1} -  x^{\alpha+1} = (\alpha+1)\int_x^y\xi^\alpha\,dx \ge (\alpha+1)(y-x)x^\alpha \ge (\alpha+1)(y-x)x.
\eeqn
Thus, we arrive at
\beqn
\begin{split}
y^{\alpha+1}\psi(y) - x^{\alpha+1}\psi(x) 
 \ge [-\|g'\|_\infty a\,m(h) + \delta'(\alpha+1)](y-x)x,
\end{split}
\eeqn
which is $\ge 0$, provided that
\beqn
\delta' \ge \frac{a}{\alpha+1}m(h)\|g'\|_\infty.
\eeqn

\noindent{\bf Step 3.} We identify a condition which guarantees $\psi(x)\le ax^{-\alpha}m(\psi)$, the latter being equivalent to~$gh(x)+\delta' \le ax^{-\alpha}(m(gh)+\delta')$. Let us assume throughout that $\|g-1\|\le\frac13$. Note that
\beqn
1-\|g-1\|_\infty \le g \le 1+\|g-1\|_\infty.
\eeqn
Thus, using $h(x) \le ax^{-\alpha}m(h)$,
\beqn
gh(x) \le  (1+\|g-1\|_\infty) ax^{-\alpha}m(h),
\eeqn
which yields
\beqn
gh(x) + \delta' \le \frac{1+ \|g-1\|_\infty}{1- \|g-1\|_\infty}ax^{-\alpha}m(gh) + \delta' = \left(1+\frac{2\|g-1\|_\infty}{1- \|g-1\|_\infty}\right)ax^{-\alpha}m(gh) + \delta',
\eeqn
We now make the assumption
\beqn
\delta' \ge \frac{4a}{a-1} m(h)\|g-1\|_\infty.
\eeqn
It yields
\beqn
\begin{split}
& \frac{2\|g-1\|_\infty}{1- \|g-1\|_\infty}ax^{-\alpha}m(gh) + \delta'
\\
& \le 2\|g-1\|_\infty\frac{1+\|g-1\|_\infty}{1- \|g-1\|_\infty}ax^{-\alpha}m(h) + x^{-\alpha}\delta' 
\\
& \le \left(4\|g-1\|_\infty a\,m(h) + \delta'\right)x^{-\alpha} \le ax^{-\alpha}\delta',
\end{split}
\eeqn
which gives the desired bound~$gh(x)+\delta' \le ax^{-\alpha}(m(gh)+\delta')$.

\noindent{\bf Final step.} Let us revert to the function~$f$ and translate the above conditions to conditions on $\lambda$, $\nu$ and $\delta$ via~\eqref{eq:g_and_f}. First of all,
\beqn
\lambda \le -\|f'\|_\infty \quad\text{and}\quad \nu\ge \|f\|_\infty-\lambda 
\quad \Longrightarrow \quad
g\ge 0 \quad\text{and}\quad g'\le 0.
\eeqn
Secondly,
\beqn
\nu \ge 3(\|f\|_\infty - \lambda)
\quad \Longrightarrow \quad 
\|g-1\|_\infty \le \frac13.
\eeqn
Finally,
\beqn
\delta \ge \frac{a}{\alpha+1}m(h)(\|f'\|_\infty-\lambda) 
\quad \Longrightarrow \quad 
\delta' \ge \frac{a}{\alpha+1}m(h)\|g'\|_\infty
\eeqn
and
\beqn
\delta \ge \frac{4a}{a-1}m(h)(\|f\|_\infty - \lambda)
\quad \Longrightarrow \quad 
\delta' \ge \frac{4a}{a-1} m(h)\|g-1\|_\infty.
\eeqn
Choosing
\beqn
\lambda = -A, \quad \nu = 6A \quad \text{and} \quad \delta = 2AB\max\!\left(\frac{a}{\alpha+1},\frac{4a}{a-1}\right)
\eeqn
is enough to satisfy all the conditions. Since
\beqn
\begin{split}
\|(f + \lambda x + \nu)h + \delta\|_1 
& \le (\|f\|_\infty + |\lambda| + \nu)m(h) + \delta_1
\\
& \le (A + A + 6A)B + 2AB\max\!\left(\frac{a}{\alpha+1},\frac{4a}{a-1}\right),
\end{split}
\eeqn
the proof is complete.
\end{proof}

Let us immediately give an example of how Lemma~\ref{lem:C^1_cone} can be put to use. We will in fact need a more general result for our purposes, in order to prove certain correlation estimates later on. But proving the special case first already reveals the basic idea.

\begin{lem}\label{lem:recursive}
Let $(T_i)_{i\ge1}$ be admissible, $f_1,f_2\in C^1([0,1])$, $h\in\cC_*$ and $n\ge 0$.
There exist $g_1,\dots,g_4\in\cC_*$ such that
\beqn
f_2\cdot \widetilde\cL_n(f_1 h) = g_1 - g_2 + g_3 - g_4
\eeqn
and
\beqn
\|g_i\|_1 \le C_1^2\|f_1\|_{C^1}\|f_2\|_{C^1}m(h).
\eeqn
\end{lem}

\begin{proof}
Let us write
\beqn
\widetilde\cL_n(f_1h) = u_1 - u_2
\eeqn
with
\beqn
u_1 = \widetilde\cL_n[(f_1 + \lambda_1 x + \nu_1)h +\delta_1]
\quad\text{and}\quad u_2 = \widetilde\cL_n[(\lambda_1 x + \nu_1)h +\delta_1].
\eeqn
With $\lambda_1$, $\nu_1$ and $\delta_1$ as in Lemma~\ref{lem:C^1_cone}, where we set $A = \|f_1\|_{C^1}$ and $B = m(h)$, the functions in the square brackets are in $\cC_*$. Since~$\cC_*$ is preserved by $\widetilde\cL_n$, also $u_1,u_2\in\cC_*$. Moreover, there exists a system constant $C_1>0$ such that
\beqn
\begin{split}
m(u_i) \le C_1\|f_1\|_{C^1}m(h), \quad i=1,2.
\end{split}
\eeqn
We now apply Lemma~\ref{lem:C^1_cone} once more, with $A=\|f_2\|_{C^1}$ and $B = C_1\|f_1\|_{C^1}m(h)$: There exist constants $\lambda_2$, $\nu_2$ and $\delta_2$ such that
\beqn
\begin{split}
f_2 u_i = [(f_2 + \lambda_2 x + \nu_2)u_i +\delta_2] - [(\lambda_2 x + \nu_2)u_i +\delta_2],  \quad i=1,2,
\end{split}
\eeqn
where the functions in the square brackets are in~$\cC_*$ and the~$L^1$ norm of each is bounded by~$C_1^2\|f_2\|_{C^1}\|f_1\|_{C^1}m(h)$.
The proof is complete.
\end{proof}

The abovementioned necessary generalization of Lemma~\ref{lem:recursive} is the following: 

\begin{thm}\label{thm:recursive}
Let $(T_i)_{i\ge1}$ be admissible, $f_1,\dots,f_{k}\in C^1([0,1])$, $h\in\cC_*$ and $0\le n_1\le \dots\le n_k$.
There exist functions $g_i\in\cC_*$ and constants $\sigma_i\in\pm1$, $1\le i\le 2^k$, such that \footnote{In order to avoid cumbersome notation involving~$k-1$ pairs of parentheses, the convention here is that each operator acts on the entire expression to its right.}
\beqn
f_k\widetilde\cL_{n_{k-1},n_{k-2}+1}\cdots f_3\widetilde\cL_{n_2,n_1+1}f_2\widetilde\cL_{n_1,1}f_1 h = \sum_{i=1}^{2^k} \sigma_i g_i
\eeqn
with
\beqn
\|g_i\|_1 \le C_1^k\|f_1\|_{C^1}\cdots\|f_k\|_{C^1}m(h).
\eeqn
\end{thm}

\begin{proof}
The proof proceeds by induction.
By Lemma~\ref{lem:recursive}, the claim holds true for $k=2$. Suppose it holds true for some fixed $k\ge 2$ and note that
\beqn
\begin{split}
& f_{k+1}\widetilde\cL_{n_k,n_{k-1}+1}\cdots f_3\widetilde\cL_{n_2,n_1+1}f_2\widetilde\cL_{n_1,1}f_1 h 
 = \sum_{i=1}^{2^k} \sigma_i f_{k+1}\cdot\widetilde\cL_{n_k,n_{k-1}+1}(g_i),
\end{split}
\eeqn
where
$
\|\widetilde\cL_{n_k,n_{k-1}+1}g_i\|_1 = \|g_i\|_1 \le C_1^k\|f_1\|_{C^1}\cdots\|f_k\|_{C^1}m(h)
$
holds. Mimicking the proof of Lemma~\ref{lem:recursive}, we apply Lemma~\ref{lem:C^1_cone} with $A=\|f_{k+1}\|_{C^1}$ and $B = C_1^k\|f_1\|_{C^1}\cdots\|f_k\|_{C^1}m(h)$: There exist constants $\lambda$, $\nu$ and $\delta$ such that
\beqn
\begin{split}
f_{k+1}\cdot \widetilde\cL_{n_k,n_{k-1}+1}g_i = [(f_{k+1} + \lambda x + \nu)\widetilde\cL_{n_k,n_{k-1}+1}g_i +\delta] - [(\lambda x + \nu)\widetilde\cL_{n_k,n_{k-1}+1}g_i +\delta]
\end{split}
\eeqn
for $1\le i\le 2^k$, where the two functions in the square brackets are in $\cC_*$ and the $L^1$ norm of each is bounded by~$C_1\|f_{k+1}\|_{C^1} C_1^k\|f_1\|_{C^1}\cdots\|f_k\|_{C^1}m(h)$. This finishes the proof.
\end{proof}


\section{Correlation decay}\label{sec:decorrelation}

Given Theorem~\ref{thm:recursive}, we are now prepared to prove the following result on correlation decay. In addition to being necessary for our immediate needs, it will be applicable to proving limit theorems beyond this paper, such as a central limit theorem.
\begin{thm}\label{thm:multi}
Let $(T_i)_{i\ge 1}$ be admissible, $f_0,f_1,\dots,f_m\in C^1([0,1])$ and $f_{m+1},f_{m+2}\dots,f_k\in L^\infty$. Moreover, let $0\le n_1\le \cdots\le n_k$. Denote
\beqn
F_m =  f_m\circ \widetilde T_{n_m}\cdots f_1\circ \widetilde T_{n_1}\cdot f_0
\eeqn
and
\beqn
G_m = f_k\circ \widetilde T_{n_k}\cdots f_{m+1}\circ \widetilde T_{n_{m+1}}.
\eeqn
Then
\beqn
\begin{split}
& \left|\int G_m F_m\,d\mu 
- \int G_m\,d\mu 
\int F_m\,d\mu\right| 
\\
& 
\le 4C_0 (2C_1)^{m+1}\!\left(\prod_{i = m+1}^k\|f_i\|_\infty\right)\! \left(\prod_{i=0}^m\|f_i\|_{C^1}\right) \! \rho(n_{m+1}-n_m)
\end{split}
\eeqn
for any probability measure $d\mu = h\, dm$ with $h\in\cC_*$.
\end{thm}

\begin{proof}
First note that, writing
\beqn
\widetilde G_m =  f_k\circ \widetilde T_{n_k,n_{m+1}+1}\cdot f_{k-1}\circ \widetilde T_{n_{k-1},n_{m+1}+1}\cdots f_{m+1},
\eeqn
we have $G_m = \widetilde G_m\circ \widetilde T_{n_{m+1}}$ and, for any constant~$c$,
\beqn
\begin{split}
& \left|\int G_m F_m\,d\mu - \int G_m\,d\mu \int F_m\,d\mu\right| = \left|\int \left(G_m-\int G_m\,d\mu\right) (F_m - c) \,d\mu\right|
\\
& = \left|\int \left(\widetilde G_m-\int G_m\,d\mu\right)\circ \widetilde T_{n_{m+1}} (F_m - c)\,d\mu\right| 
\\
& = \left|\int \left(\widetilde G_m-\int G_m\,d\mu\right)\widetilde\cL_{n_{m+1}} (F_m  h - c  h)\,dx\right|
\\
& \le 2\|\widetilde G_m\|_\infty \|\widetilde\cL_{n_{m+1}} (F_m  h - c h)\|_1 
\\
& \le 2\!\left(\prod_{i = m+1}^k\|f_i\|_\infty\right)\! \|\widetilde\cL_{n_{m+1},n_m+1}(\widetilde\cL_{n_m} (F_m  h) - c \tilde h)\|_1
\end{split}
\eeqn
where we have introduced $\tilde h = \widetilde\cL_{n_m}h\in\cC_*$. 
Choosing
\beqn
c = \int \widetilde\cL_{n_m}(F_m  h)\,dx = \int F_m h\,dx
\eeqn
guarantees
\beqn
\int (\widetilde\cL_{n_m} (F_m  h) - c  \tilde h)\, dx = 0.
\eeqn
Using the basic identity $\widetilde\cL_n(f\circ \widetilde T_n\, g) = f \widetilde\cL_n g$, note that
\beqn
\widetilde\cL_{n_i}(F_i  h) = \widetilde\cL_{n_i}(f_i\circ \widetilde T_{n_i} F_{i-1}  h) = f_i\widetilde\cL_{n_i,n_{i-1}+1}\widetilde\cL_{n_{i-1}}(F_{i-1} h).
\eeqn
By induction,
\beqn
\begin{split}
\widetilde\cL_{n_m} (F_m  h) 
& = f_m\widetilde\cL_{n_m,n_{m-1}+1}\cdots f_2\widetilde\cL_{n_2,n_1+1}\widetilde\cL_{n_1}(F_1  h)
\\
& = f_m\widetilde\cL_{n_m,n_{m-1}+1}\cdots f_2\widetilde\cL_{n_2,n_1+1} f_1\widetilde\cL_{n_1} f_0  h.
\end{split}
\eeqn
We can now apply Theorem~\ref{thm:recursive}, according to which
\beqn
\widetilde\cL_{n_m}(F_m  h) =  \sum_{i=1}^{2^{m+1}} \sigma_i g_i
\eeqn
for suitable functions $g_i\in\cC_*$ and constants $\sigma_i\in\pm1$, $1\le i\le 2^{m+1}$, satisfying
\beqn
\|g_i\|_1 \le C_1^{m+1}\|f_0\|_{C^1}\cdots\|f_m\|_{C^1}.
\eeqn
Writing $c_i = \int g_i\,dx$,
\beqn
\begin{split}
\widetilde\cL_{n_m} (F_m   h) - c  \tilde h = \sum_{i=1}^{2^{m+1}} \sigma_i (g_i - c_i  \tilde h).
\end{split}
\eeqn
Each function in the parentheses on the right side satisfies the assumptions of Lemma~\ref{lem:Aimino}. This yields
\beqn
\begin{split}
\|\widetilde\cL_{n_{m+1},n_m+1} (\widetilde\cL_{n_m} (F_m  h) - c  \tilde h)\|_1 & \le \sum_{i=1}^{2^{m+1}}\|\widetilde\cL_{n_{m+1},n_m+1}(g_i - c_i  \tilde h)\|_1 
\\
& \le C_0\sum_{i=1}^{2^{m+1}}(\|g_i\|_1 + |c_i|)\rho(d_m)
\\
& \le 2C_0\sum_{i=1}^{2^{m+1}}\|g_i\|_1\rho(d_m) 
\\
& \le 2C_0 2^{m+1} C_1^{m+1}\|f_0\|_{C^1}\cdots\|f_m\|_{C^1} \rho(d_m).
\end{split}
\eeqn
Collecting the estimates proves the theorem.
\end{proof}



\medskip
\section{Perturbation of transfer operator and SRB density}\label{sec:perturbation}
Here we prove that the SRB density $\hat h_\alpha$ and the pushforward $\cL_\alpha h$ of any initial density~$h\in\cC_*$ depend H\"older continuously on the parameter~$\alpha$ in the $L^1$ norm. To the authors' knowledge these are new results.

\begin{thm}\label{thm:SRB_continuous}
Let $0<\beta_*<1$. There exists a constant $C_2 = C_2(\beta_*)>0$ such that 
\beq\label{eq:transfer_perturbation}
\|(\cL_\alpha - \cL_\beta)h\|_1 \le C_2\|h\|_1(\beta-\alpha)^{\frac13(1- \beta_*)}|{\log(\beta-\alpha)}| \qquad \forall\, h\in\cC_*
\eeq
and
\beq\label{eq:SRB_perturbation}
\|\hat h_\alpha-\hat h_\beta\|_1 \le C_2(\beta-\alpha)^{\frac13(1- \beta_*)^2} |{\log (\beta-\alpha)}|^\frac1{\beta_*}
\eeq
hold whenever $0\le\alpha<\beta\le \beta_*$.
\end{thm}

We do not claim the result to be optimal regarding the exponents.

From Theorem~\ref{thm:SRB_continuous} we immediately get
\begin{cor}\label{cor:SRB_mean_continuous}
Let $0<\beta_*<1$ and let $C_2>0$ be as in Theorem~\ref{thm:SRB_continuous}. For every bounded measurable function $f:[0,1]\to\bR$,
\beqn
\left|\int f\,d\hat\mu_\alpha - \int f\,d\hat\mu_\beta\right| \le C_2\|f\|_\infty|(\beta-\alpha)^{\frac13(1- \beta_*)^2} |{\log (\beta-\alpha)}|^\frac1{\beta_*}
\eeqn
holds whenever $0\le\alpha< \beta\le \beta_*$.
\end{cor}

\begin{proof}[Proof of Theorem~\ref{thm:SRB_continuous}]
The proof is given in the same order as the claims in the theorem.

\smallskip
\noindent{\bf Part 1: perturbation of the transfer operator.}
Let us first assume $\|h\|_1 = 1$.
Since the maps $T_\alpha$ and $T_\beta$ agree on $[\frac12,1]$,
\beqn
(\cL_\alpha - \cL_\beta) h(x) = \frac{h(y_\alpha)}{T_\alpha'(y_\alpha)} - \frac{h(y_\beta)}{T_\beta'(y_\beta)},
\eeqn
where $y_\alpha$ is the preimage of $x$ under the left branch $T_\alpha|_{[0,\frac12)}$ and $y_\beta$ is defined similarly.
We begin by fixing $\ve\in(0,1)$ and noting that
\beqn
\int_{[0,\ve]}\frac{h(y_\alpha)}{T_\alpha'(y_\alpha)}\,dx = \int_{\{y:T_\alpha(y)\in[0,\ve]\}} h(y)\,dy \le \int_{[0,\ve]} h(y)\,dy \le C\ve^{1-\beta_*},
\eeqn
where the last inequality uses the fact that \(h \in \cC_*\), and the constant \(C\) depends only on~\(\beta_*\). A similar bound holds for $y_\beta$ in place of $y_\alpha$ and $T_\beta$ in place of $T_\alpha$. Hence,
\beqn
\int_{[0,\ve]}\left|\frac{h(y_\alpha)}{T_\alpha'(y_\alpha)} - \frac{h(y_\beta)}{T_\beta'(y_\beta)}\right|dx 
 \le 2C\ve^{1-\beta}.
\eeqn
We are left with an $L^1$ estimate on $(\ve,1]$. To that end, we bound
\beqn
\begin{split}
\left|\frac{h(y_\alpha)}{T_\alpha'(y_\alpha)} - \frac{h(y_\beta)}{T_\beta'(y_\beta)}\right|
& = \left|\frac{h(y_\alpha)-h(y_\beta)}{T_\alpha'(y_\alpha)} - \frac{h(y_\beta)}{T_\beta'(y_\beta)} \frac{1}{T_\alpha'(y_\alpha)}\left(T_\alpha'(y_\alpha) - T_\beta'(y_\beta) \right)\right|
\\
& \le |h(y_\alpha)-h(y_\beta)| + \frac{h(y_\beta)}{T_\beta'(y_\beta)} |T_\alpha'(y_\alpha) - T_\beta'(y_\beta)|.
\end{split}
\eeqn
Observe that 
\beqn
T_\alpha'(y_\alpha) = (1+\alpha)xy_\alpha^{-1} - \alpha,
\eeqn
where
\beqn
1\le xy_\alpha^{-1}\le 2,
\eeqn
because $x = T_\alpha(y_\alpha) = y_\alpha(1+2^{\alpha}y_{\alpha}^{\alpha})$.
Using this together with $\|\cL_\beta h\|_1 = 1$, we estimate
\beqn
\begin{split}
& \int_{(\ve,1]} \frac{h(y_\beta)}{T_\beta'(y_\beta)} |T_\alpha'(y_\alpha) - T_\beta'(y_\beta)|\,dx 
\\
& \le \|\cL_\beta h\|_1\|((1+\alpha)xy_\alpha^{-1} - \alpha - (1+\beta)xy_\beta^{-1} + \beta)1_{(\ve,1]}(x)\|_\infty 
\\
& \le (\beta-\alpha) + \|((1+\alpha)xy_\alpha^{-1} - (1+\beta)xy_\beta^{-1})1_{(\ve,1]}(x)\|_\infty 
\\
& = (\beta-\alpha) + \|((1+\alpha)(xy_\alpha^{-1} - xy_\beta^{-1}) - (\beta-\alpha)xy_\beta^{-1})1_{(\ve,1]}(x)\|_\infty 
\\
& \le 3(\beta-\alpha) + 2\|(xy_\alpha^{-1} - xy_\beta^{-1})1_{(\ve,1]}(x)\|_\infty 
\\
& \le 3(\beta-\alpha) + 2\|xy_\alpha^{-1}xy_\beta^{-1}x^{-1}(y_\alpha - y_\beta)1_{(\ve,1]}(x)\|_\infty 
\\
& \le 3(\beta-\alpha) + 8\|x^{-1}(y_\alpha - y_\beta)1_{(\ve,1]}(x)\|_\infty 
\\
& \le 3(\beta-\alpha) + 8\ve^{-1}\|(y_\alpha - y_\beta) 1_{(\ve,1]}(x)\|_\infty.
\end{split}
\eeqn
Note that, because $T_\beta$ is expanding,
\beqn
|y_\beta-y_\alpha| \le |T_\beta(y_\beta) - T_\beta(y_\alpha)| = |T_\alpha(y_\alpha) - T_\beta(y_\alpha)|.
\eeqn 
On the other hand, $x\ge \ve$ implies $y_\alpha,y_\beta\ge \frac{\ve}{2}$, so
\beqn
\begin{split}
|T_\alpha(y_\alpha)-T_\beta(y_\alpha)| 
& = \bigl | 2^{\alpha}y_{\alpha}^{\alpha+1} - 2^{\beta}y_{\alpha}^{\beta+1}\bigr| 
=\tfrac12 \bigl| (2y_{\alpha})^{1+ \alpha} - (2y_{\alpha})^{1+ \beta}\bigr| \\
& \le \bigl|\log(2 y_{\alpha}) \bigr| |\beta - \alpha| \le  \log (\ve^{-1})(\beta - \alpha)
\end{split}
\eeqn
The function $u(x) = x^{1+\beta_*}h(x)$ is Lipschitz continuous with a Lipschitz constant $C$ depending only on \(\beta_*\).\footnote{See the footnote in the proof of Lemma~2.3 in~\cite{LiveraniSaussolVaienti_1999}.} For $x\ge \ve$ (and $y_\alpha,y_\beta\ge \frac{\ve}{2}$) this implies
\beqn
\begin{split}
& |h(y_\alpha) - h(y_\beta)|
\\
& = |y_\alpha^{-1-\beta_*}u(y_\alpha) - y_\beta^{-1-\beta_*}u(y_\beta)|
\\
& \le |(y_\alpha^{-1-\beta*} - y_\beta^{-1-\beta*}) u(y_\alpha)| + |y_\beta^{-1-\beta}(u(y_\beta) - u(y_\alpha))|
\\
& \le (C\ve^{-2-\beta}+C\ve^{-1-\beta})|y_\beta-y_\alpha| 
\\
& \le C\ve^{-2-\beta}|y_\beta-y_\alpha|,
\end{split}
\eeqn 
where $C>0$ still depends only on $\beta_*$. We now have (with new constants)
\beqn
\begin{split}
\|(\cL_\alpha - \cL_\beta) h\|_1 
& \le C(\ve^{1-\beta} + (\beta-\alpha)\ve^{-2-\beta}\log\ve^{-1})
\le C(\ve^{1-\beta} + (\beta-\alpha)\ve^{-2-\beta})\log\ve^{-1},
\end{split}
\eeqn
and setting $\ve = (\beta-\alpha)^{\frac13}$ yields
\beqn
\|(\cL_\alpha - \cL_\beta) h\|_1  \le C(\beta-\alpha)^{\frac13(1- \beta)}|{\log(\beta-\alpha)}| \le C(\beta-\alpha)^{\frac13(1- \beta_*)}|{\log(\beta-\alpha)}| .
\eeqn
The case $\|h\|_1\ne 1$ is recovered by scaling, which finishes the proof of the first part.

\smallskip
\noindent{\bf Part 2: perturbation of the SRB density.}
Since $\cL_\alpha\hat h_\alpha = \hat h_\alpha$ for each $\alpha$, we have, for any $n\ge 0$,
\beqn
\hat h_\alpha-\hat h_\beta = \cL_\alpha^n (\hat h_\alpha-1) - \cL_\alpha^n (\hat h_\beta-1) + (\cL_\alpha^n - \cL_\beta^n) \hat h_\beta.
\eeqn
Here $\hat h_\alpha$, $\hat h_\beta$ and $1$ are in $\cC_*$ with $\int \hat h_\alpha\,dx = \int \hat h_\beta\,dx = 1$, so Lemma~\ref{lem:Aimino} yields
\beqn
\begin{split}
\|\cL_\alpha^n (\hat h_\alpha-1) - \cL_\alpha^n (\hat h_\beta-1)\|_1 
& \le 4C_0n^{-(\frac1{\beta_*}-1)}(\log n)^\frac1{\beta_*},
\end{split}
\eeqn
where the constant depends only on $\beta_*$.
Since
\beqn
(\cL_\alpha^n - \cL_\beta^n) \hat h_\beta = \sum_{k=1}^n \cL_\alpha^{k-1}(\cL_\alpha - \cL_\beta)\cL_\beta^{n-k} \hat h_\beta = \sum_{k=1}^n \cL_\alpha^{k-1}(\cL_\alpha - \cL_\beta) \hat h_\beta
\eeqn
and $\cL_\alpha$ is an $L^1$ contraction, we also have
\beqn
\|(\cL_\alpha^n - \cL_\beta^n) \hat h_\beta\|_1 \le n\|(\cL_\alpha - \cL_\beta) \hat h_\beta\|_1.
\eeqn

Applying the first part of the theorem and collecting all the estimates, we arrive at
\beqn
\begin{split}
\|\hat h_\alpha-\hat h_\beta\|_1 
& \le Cn(n^{-\frac1{\beta_*}}(\log n)^\frac1{\beta_*} + (\beta-\alpha)^{\frac13(1- \beta_*)}|{\log(\beta-\alpha)}|)
\\
& \le Cn(n^{-\frac1{\beta_*}} + (\beta-\alpha)^{\frac13(1- \beta_*)})\max((\log n)^\frac1{\beta_*},|{\log(\beta-\alpha)}|),
\end{split}
\eeqn
where $C$ depends on $\beta_*$ only.
Choosing $n = \lfloor(\beta-\alpha)^{-\frac13\beta_*(1- \beta_*)}\rfloor \ge 1$ yields
\beqn
\begin{split}
\|\hat h_\alpha-\hat h_\beta\|_1
& \le C(\beta-\alpha)^{-\frac13\beta_*(1- \beta_*)}(n^{-\frac1{\beta_*}} + (\beta-\alpha)^{\frac13(1- \beta_*)})\max((\log n)^\frac1{\beta_*},|{\log(\beta-\alpha)}|)
\\ 
& \le C(\beta-\alpha)^{\frac13(1- \beta_*)^2}\max(|{\log (\beta-\alpha)}|^\frac1{\beta_*},|{\log(\beta-\alpha)}|)
\\
& \le C(\beta-\alpha)^{\frac13(1- \beta_*)^2} |{\log (\beta-\alpha)}|^\frac1{\beta_*}
\end{split}
\eeqn
with~$C$ depending on $\beta_*$ only, where the last estimate uses $\beta-\alpha\le \beta_*<1$. This finishes the proof of the second part.

The proof of Theorem~\ref{thm:SRB_continuous} is now complete.
\end{proof}


\section{Proof of main theorem}\label{sec:proof_main}

We are now ready to enter the proof of Theorem~\ref{thm:main}. We begin by recording certain facts concerning general values $\beta_*\in(0,1)$, which are essential for the proof, in Section~\ref{sec:proof_main_1}. Parts (i) and (ii) of the theorem are proved in Sections~\ref{sec:proof_main_2} and~\ref{sec:proof_main_3}, respectively.


\subsection{Preliminaries ($0<\beta_*<1$)}\label{sec:proof_main_1} Notice first that it will be enough to establish the theorem for all~$f\in C^1([0,1])$: Since~$C^1([0,1])$ is a subalgebra of~$C([0,1])$ which contains the constant functions and separates points, it is dense by the Stone--Weierstrass theorem. Hence, given $g\in C([0,1])$ and~$\ve>0$, there exists~$f\in C^1([0,1])$ such that
\beqn
\sup_{t\in[0,1]}\left|\int_0^{t} g_{n,\round{ns}}(x)\,d s - \int_0^t \hat\mu_s(g)\,d s\right| \le \sup_{t\in[0,1]}\left|\int_0^{t} f_{n,\round{ns}}(x)\,d s - \int_0^t \hat\mu_s(f)\,d s\right| + \ve
\eeqn
holds for all~$x$ and all~$n$.

Let $\{I_1,\dots\,I_m\}$ be the regularity partition of $[0,1]$ associated to~$\gamma$. We may assume $0=\tau_1<\tau_2<\dots <\tau_m<\tau_{m+1}=1$, where $\tau_\ell$ and $\tau_{\ell+1}$ are the endpoints of the interval~$I_\ell$. There exists a constant $C_\gamma>0$ such that
\beq\label{eq:gamma_Holder}
|\gamma_t - \gamma_s| \le C_\gamma|t-s|^\theta
\eeq
holds for all $\tau_\ell < s,t < \tau_{\ell+1}$ and all $\ell$, and
\beqn
\sup_{t\in[0,1]}|\alpha_{n,\round{nt}} - \gamma_t| \le C_\gamma n^{-\theta}
\eeqn
holds for all $n\ge 1$. Since
$
|\alpha_{n,k} - \alpha_{n,j}| \le |\alpha_{n,nkn^{-1}} - \gamma_{kn^{-1}}| + |\alpha_{n,njn^{-1}} - \gamma_{jn^{-1}}| + |\gamma_{kn^{-1}} - \gamma_{jn^{-1}}|
$,
the bounds above imply that
\beq\label{eq:alpha_distance}
|\alpha_{n,k} - \alpha_{n,j}| \le C_\gamma n^{-\theta}(2 + |j-k|^\theta)
\eeq
whenever $n\tau_\ell < j,k < n\tau_{\ell+1}$ holds for some~$\ell$.

Note that there exist $0<\beta_*<1$ such that $\gamma([0,1])\subset[0,\beta_*]$ and $\alpha_{n,k}\in[0,\beta_*]$ for $0\le k\le n$ and sufficiently large values of $n$. Since we are only interested in the limit $n\to\infty$, we will assume --- without loss of generality --- that the latter condition holds starting from $n=1$.

We will shortly need the next lemma. For brevity, denote $\cL_{n,k}=\cL_{\alpha_{n,k}}$ and $\hat h_{n,k} = \hat h_{\alpha_{n,k}}$.

\begin{lem}\label{lem:pushforward_convergence}
Let $\mu$ be a probability measure with density $h\in\cC_*$, and let $\mu_{n,k}$ be its pushforward with density $h_{n,k} = \cL_{n,k}\cdots\cL_{n,1} h$. There exist constants $c_0,c_1>0$ and $p_0,p_1\in(0,1)$ such that
\beq\label{eq:pushforward_convergence_1}
\|h_{n,k} - \hat h_{n,k}\|_1 \le c_0 n^{-p_0}
\eeq
whenever $n(\tau_\ell + c_1n^{p_1-1}) < k < n\tau_{\ell+1}$ holds for some~$\ell$. Moreover, given a bounded function $f:[0,1]\to\bR$,
\beq\label{eq:pushforward_convergence_2}
|\mu(f_{n,\round{nt}}) - \hat\mu_{\gamma_t}(f)| \le c_0\|f\|_\infty n^{-p_0}
\eeq
whenever $\tau_\ell + c_1n^{p_1-1} < t < \tau_{\ell+1}$ holds for some~$\ell$.
\end{lem}

For the purposes of this paper, the error rate in~\eqref{eq:pushforward_convergence_2} is unnecessary. Moreover,~\eqref{eq:pushforward_convergence_1} is only used in the proof of~\eqref{eq:pushforward_convergence_2}. However, both bounds will be useful for establishing finer statistical properties of the intermittent QDS later on. Information on the system-parameter dependence of the constants can be extracted from the proof below, which we leave to the interested reader.

\begin{proof}[Proof of Lemma~\ref{lem:pushforward_convergence}]
Since $\cL_{n,k}\hat h_{n,k} = \hat h_{n,k}$, we have
\beqn
\begin{split}
 h_{n,k} - \hat h_{n,k} & = 
\cL_{n,k}\cdots \cL_{n,k-K+1}( h_{n,k-K}-\hat h_{n,k})
+ (\cL_{n,k}\cdots \cL_{n,k-K+1}-\cL_{n,k}^{K})\hat h_{n,k} 
\end{split}
\eeqn
whenever $1\le K<k$.
In order to bound $h_{n,k} - \hat h_{n,k}$ in $L^1$, note that Lemma~\ref{lem:Aimino} implies
\beqn
\| \cL_{n,k}\cdots \cL_{n,k-K+1}(h_{n,k-K}-\hat h_{n,k})  \|_1 \le C_0(\|h_{n,k-K}\|_1+\|\hat h_{n,k}\|_1) \rho(K) \le 2C_0\rho(K)
\eeqn
Recalling that $\cL_\alpha$ is an $L^1$ contraction and using~\eqref{eq:transfer_perturbation} in Theorem~\ref{thm:SRB_continuous},
\beqn
\begin{split}
& \|(\cL_{n,k}\cdots \cL_{n,k-K+1}-\cL_{n,k}^{K})\hat h_{n,k}\|_1 
\\
& = \left\| \sum_{j=k-K+1}^k \cL_{n,k}\cdots \cL_{n,j+1}(\cL_{n,j}-\cL_{n,k})\cL_{n,k}^{j - (k-K+1)}\hat h_{n,k} \right\|_1
\\
& \le \sum_{j=k-K+1}^k \| (\cL_{n,j}-\cL_{n,k})\hat h_{n,k} \|_1
 \le K \max_{k-K+1\le j\le k} \| (\cL_{n,j}-\cL_{n,k})\hat h_{n,k} \|_1
\\
& \le C_2 K \max_{k-K+1\le j\le k} |\alpha_{n,j}-\alpha_{n,k}|^{\frac13(1- \beta_*)}|{\log|\alpha_{n,j}-\alpha_{n,k}|}|
\\
& \le C K \max_{k-K+1\le j\le k} |\alpha_{n,j}-\alpha_{n,k}|^{\frac14(1- \beta_*)}
\end{split}
\eeqn
for a constant depending on $\beta_*$. Assuming
\beq\label{eq:Kk}
n\tau_\ell + K < k < n\tau_{\ell+1}
\eeq
for some $\ell$, we can now use~\eqref{eq:alpha_distance}, which yields
\beqn
\begin{split}
& \|(\cL_{n,k}\cdots \cL_{n,k-K+1}-\cL_{n,k}^{K})\hat h_{n,k}\|_1 
\le C K [C_\gamma n^{- \theta}(2 + K^ \theta)]^{\frac14(1- \beta_*)}
\\
& \le C K [3C_\gamma n^{- \theta}K^ \theta]^{\frac14(1- \beta_*)} \le Cn^{-\theta\frac14(1- \beta_*)}K^{1+\theta\frac14(1- \beta_*)}
\end{split}
\eeqn
for constants $C$ depending only on $\beta_*$ and $\theta$ and $\gamma$. Writing $\kappa = \frac14(1- \beta_*)$,
we thus get
\beqn
\begin{split}
\| h_{n,k} - \hat h_{n,k}\|_1 
& \le 2C_0K^{-(\frac1{\beta_*}-1)}(\log K)^{\frac1{\beta_*}} + C n^{- \theta\kappa}K^{1+ \theta\kappa}
\\
& \le CK\!\left(K^{-\frac1{\beta_*}}(\log K)^{\frac1{\beta_*}} + n^{- \theta\kappa}K^{\theta\kappa}\right)
\end{split}
\eeqn
whenever~\eqref{eq:Kk} holds. Fixing $K = \lceil n^{ \theta\kappa (\frac1{\beta_*} +  \theta\kappa)^{-1}}\rceil$ (so that $K^{-\frac1{\beta_*}} \approx  n^{- \theta\kappa}K^{ \theta\kappa}$) yields
\beqn
\| h_{n,k} - \hat h_{n,k}\|_1 \le Cn^{ \theta\kappa (\frac1{\beta_*} +  \theta\kappa)^{-1}(1-\frac1{\beta_*})} (\log n)^{\frac1{\beta_*}} \le c_0n^{-p_0},
\eeqn
where $p_0$ depends on $\theta$ and $\beta_*$, and $c_0$ depends additionally on~$\gamma$. This proves the first claim of the lemma, the values of $c_1$ and $p_1$ being determined by the choice of~$K$.

In order to prove second claim, note first that $\mu(f_{n,\round{nt}}) = \mu_{n,\round{nt}}(f)$. Thus it suffices to bound $\|h_{n,\round{nt}} - \hat h_{\gamma_t}\|_1$. But, for $k = \round{nt}$ satisfying~\eqref{eq:Kk}, we have
\beqn
\begin{split}
\|h_{n,\round{nt}} - \hat h_{\gamma_t}\|_1 & \le \|h_{n,\round{nt}} - \hat h_{n,\round{nt}}\|_1 + \|\hat h_{n,\round{nt}} - \hat h_{\gamma_t}\|_1
\\
& \le c_0n^{-p_0} + C|\alpha_{n,\round{nt}}-\gamma_t|^{\frac14(1-\beta_*)^2}
\\
& \le c_0n^{-p_0} + C n^{- \theta\frac14(1-\beta_*)^2},
\end{split}
\eeqn
with the aid of~\eqref{eq:SRB_perturbation}. The constant $C$ depends only on $\beta_*$ and $\theta$ and $\gamma$. Redefining~$c_0$ and~$p_0$ proves also the second claim. 
\end{proof}

\pagebreak
Our task is to bound
\beqn
\zeta_n(x,t) - \zeta(t) = \bar\zeta_{n}(x,t) + \int_0^t \mu(f_{n,\round{ns}}) - \hat\mu_{\gamma_s}(f) \,ds
\eeqn
uniformly in $t$, where
\beqn
\bar\zeta_{n}(x,t) = \int_0^t \bar f_{n,\round{ns}}(x) \,ds
\quad\text{and}\quad
\bar f_{n,\round{ns}}(x) = f_{n,\round{ns}}(x) - \mu(f_{n,\round{ns}}).
\eeqn
Since $f$ is bounded, there are no integrability issues; in particular, Corollary~\ref{cor:SRB_mean_continuous} and the regularity of the curve~$\gamma$ imply that $s\mapsto\hat\mu_{\gamma_s}(f)$ is piecewise continuous. We can proceed with an application of~\eqref{eq:pushforward_convergence_2} in Lemma~\ref{lem:pushforward_convergence}:
\beqn
\begin{split}
& \sup_{t\in[0,1]}\left|\int_0^t \mu(f_{n,\round{ns}}) - \hat\mu_{\gamma_s}(f) \,ds\right| 
\le \int_0^1 |\mu(f_{n,\round{ns}}) - \hat\mu_{\gamma_s}(f)| \,ds
\\
& \qquad = \sum_{\ell=1}^m\int_{\tau_\ell}^{\tau_\ell+c_1n^{p_1-1}} |\mu(f_{n,\round{ns}}) - \hat\mu_{\gamma_s}(f)| \,ds 
+ \sum_{\ell=1}^m \int_{\tau_\ell+c_1n^{p_1-1}}^{\tau_{\ell+1}} |\mu(f_{n,\round{ns}}) - \hat\mu_{\gamma_s}(f)| \,ds
\\
& \qquad \le 2\|f\|_\infty m c_1n^{p_1-1} +  c_0\|f\|_\infty n^{-p_0},
\end{split}
\eeqn
where the last bound tends to zero with increasing~$n$. 

Thus, it remains to bound $\bar\zeta_n$ in what follows.
For convenience, we write $\bar\zeta_{n}(t)$ for the function $x\mapsto\bar\zeta_{n}(x,t)$. Roughly, the common strategy in the two cases of the theorem is to obtain moment bounds on $\zeta_n(t)$, for fixed values of~$t$. The uniform Lipschitz continuity 
\beqn
|\bar\zeta_n(x,t) - \bar\zeta_n(x,s)| = \left|\int_s^t \bar f_{n,\round{nr}}\,dr\right| \le 2\|f\|_\infty |t-s|
\eeqn
is then exploited for dealing with the uncountably many values~$t$ can assume. 


\subsection{Case $\frac12\le \beta_*<1$: proof of part (i)}\label{sec:proof_main_2}

This case is based on the following second moment bound:

\begin{lem}
Suppose $\beta_*\in[\frac12,1)$. Given $\delta \in (0,\frac1{\beta_*}-1)$, there exists a constant $C_* = C_*(\beta_*,\delta)>0$ such that
\beq\label{eq:L^2-bound}
\mu\!\left(|\bar\zeta_{n}(t)|^2\right) \le C_*\|f\|_\infty\|f\|_{C^1} n^{-\delta} t^{2-\delta}
\eeq
for all $n$ and $t$.\footnote{For completeness, we remark that if $\beta_*\in(0,\frac12)$, the corresponding bound is $C(\beta_*)\|f\|_\infty\|f\|_{C^1} n^{-1} t$, which follows from the summability of $\rho$.}
\end{lem}
\begin{proof}
Fix $\delta \in (0,\frac1{\beta_*}-1)$ and define the auxiliary function
\beqn
\phi(u) = 
\begin{cases}
u^{-\delta}, & u\ge 2 \\
1, & 0\le u < 2.
\end{cases}
\eeqn
There exists a constant $C_\phi = C_\phi(\beta_*,\delta)$ such that $\rho(n)\le C_\phi\phi(n)$ holds for all $n\ge 0$. Moreover, elementary computations show that
\beqn
\phi(\round{v}-\round{u}) \le \phi(\tfrac12(v-u)) \le 8^\delta\phi(v-u), \quad u\le v.
\eeqn
Thus, by Theorem~\ref{thm:multi},
\beqn
\begin{split}
\mu(|\bar\zeta_n(t)|^2) & =  \int_0^t \! \int_0^t \mu(\bar f_{n,\round{ns}}\bar f_{n,\round{nr}})\,dr\,ds = 2 \int_0^t \! \int_0^s \mu(\bar f_{n,\round{ns}}\bar f_{n,\round{nr}})\,dr\,ds
\\
& \le 64C_0C_1\int_0^t \! \int_0^s\|f-\mu(f_{n,\round{ns}})\|_\infty\|f-\mu(f_{n,\round{nr}})\|_{C^1}  \rho(\round{ns}-\round{nr}) \,d r\,d s
\\
& \le 64C_0C_1 8^\delta C_\phi 4 \|f\|_\infty\|f\|_{C^1} \int_0^t \! \int_0^s \phi(ns-nr) \,d r\,d s
\\
& = 256C_0C_1 8^\delta C_\phi\|f\|_\infty\|f\|_{C^1} \int_0^t \! \int_0^s \phi(nr) \,d r\,d s
\\
& = 256C_0C_1 8^\delta C_\phi\|f\|_\infty\|f\|_{C^1}\cdot n^{-1} \int_0^t \! \int_0^{ns} \phi(r) \,d r\,d s
\\
& \le 256C_0C_1 8^\delta C_\phi\|f\|_\infty\|f\|_{C^1}\cdot n^{-1} t \int_0^{nt} \phi(r) \,d r
\end{split}
\eeqn
Assuming $nt\ge 2$,
\beqn
\begin{split}
\int_0^{nt} \phi(r) \,d r = 2 + \tfrac1{1-\delta}((nt)^{1-\delta}-2^{1-\delta}) \le 2 + \tfrac1{1-\delta}(nt)^{1-\delta} \le (2^\delta + \tfrac1{1-\delta})(nt)^{1-\delta}.
\end{split}
\eeqn
The same bound applies in the case $0\le nt<2$, as
\beqn
\int_0^{nt} \phi(r) \,d r = nt = (nt)^\delta (nt)^{1-\delta} \le 2^\delta (nt)^{1-\delta} \le (2^\delta + \tfrac1{1-\delta})(nt)^{1-\delta}.
\eeqn
Collecting the bounds proves the lemma.
\end{proof}

Let $(K_n)_{n\ge 1}$ be an increasing sequence of positive integers, to be fixed later. Define the finite parameter sets
\beqn
\cT_n = \{jK_n^{-1}\,:\,j=1,\dots,K_n\}, \quad n\ge 1.
\eeqn
Given $t\in[0,1]$, let $t_n$ be the smallest element of~$\cT_n$ larger than~$t$. Then
\beqn
\begin{split}
\sup_{t\in[0,1]}|\bar\zeta_{n}(t)| 
& \le \sup_{t\in\cT_n}|\bar\zeta_{n}(t)| + \sup_{t\in[0,1]}|\bar\zeta_{n}(t) - \bar\zeta_{n}(t_n)| 
\\
& \le  \sup_{t\in\cT_n}|\bar\zeta_{n}(t)| + 2\|f\|_\infty K_n^{-1}.
\end{split}
\eeqn
Since $K_n$ is increasing, we thus have
\beq\label{eq:finite}
\limsup_{n\to\infty} \sup_{t\in[0,1]}|\bar\zeta_{n}(x,t)| \le \limsup_{n\to\infty}  \sup_{t\in\cT_n}|\bar\zeta_{n}(x,t)|
\eeq
for every~$x$. Next, following~\cite{Lyons_1988}, we recall a Cauchy condensation criterion for the summability of a subsequence.
\begin{lem}\label{lem:Cauchy}
Let the numbers $a_n\ge 0$, $n\ge 1$, satisfy
\beqn
\sum_{n=1}^\infty \frac{a_n}n < \infty.
\eeqn
Then there exists an increasing sequence $(n_k)_{k\ge 1}$ such that
\beqn
\sum_{k = 1}^\infty a_{n_k} < \infty
\quad\text{and}\quad
\lim_{k\to\infty}\frac{n_{k+1}}{n_k} = 1.
\eeqn
\end{lem}
We implement the lemma with $a_n = \mu((\sup_{t\in\cT_n}|\bar\zeta_{n}(t)|)^2)$.
Observe that
\beqn
\mu\!\left(\left(\sup_{t\in\cT_n}|\bar\zeta_{n}(t)|\right)^2\right) = \mu\!\left(\sup_{t\in\cT_n}|\bar\zeta_{n}(t)|^2\right) \le \sum_{t\in\cT_n} \mu\!\left(|\bar\zeta_{n}(t)|^2\right)
\eeqn
together with \eqref{eq:L^2-bound} yields
\beqn
\begin{split}
\sum_{n=1}^\infty \frac1n \mu\!\left(\left(\sup_{t\in\cT_n}|\bar\zeta_{n}(t)|\right)^2\right) 
&\le C_*\|f\|_\infty\|f\|_{C^1} \sum_{n=1}^\infty \frac1{n^{1+\delta}} \sum_{t\in\cT_n} t^{2-\delta}
\\
&\le C_*\|f\|_\infty\|f\|_{C^1} \sum_{n=1}^\infty \frac {K_n}{n^{1+\delta}}.
\end{split}
\eeqn
Choosing, say, $K_n = \round{\log n}$ guarantees that $\sum_{n=1}^\infty  K_n/n^{1+\delta}<\infty$. By Lemma~\ref{lem:Cauchy},
\beqn
\mu\!\left(\sum_{k=1}^\infty \left(\sup_{t\in\cT_{n_k}}|\bar\zeta_{n_k}(t)|\right)^2\right) = \sum_{k=1}^\infty \mu\!\left(\left(\sup_{t\in\cT_{n_k}}|\bar\zeta_{n_k}(t)|\right)^2\right) < \infty,
\eeqn
for some increasing sequence $(n_k)_{k\ge 1}$ satisfying $\lim_{k\to\infty}n_{k+1}/n_k = 1$. This implies that
\beqn
\lim_{k\to\infty}\sup_{t\in\cT_{n_k}}|\bar\zeta_{n_k}(x,t)| = 0
\eeqn
for almost every~$x$ with respect to~$\mu$. In view of~\eqref{eq:finite},
\beqn
\lim_{k\to\infty}\sup_{t\in [0,1]}|\bar\zeta_{n_k}(t)| = 0
\eeqn
almost everywhere.

Finally, recall that a sequence of random variables converges in probability if and only if every subsequence has a further subsequence which converges almost surely. Repeating the above proof, mutatis mutandis, for an arbitrary subsequence $\sup_{t\in [0,1]}|\bar\zeta_{n_k}(t)|$, $k\ge 1$, yields a further subsequence $\sup_{t\in [0,1]}|\bar\zeta_{n_{k_j}}(t)|$, $j\ge 1$, which converges almost surely.

The proof of part (i) of Theorem~\ref{thm:main} is now complete. \qed


\subsection{Case $\beta_*<\frac12$: proof of part (ii)}\label{sec:proof_main_3} In this case we cannot resort to the same method as in the previous one, because we want to prove almost sure convergence of the entire sequence, not just a subsequence. Following the approach of~\cite{Stenlund_2015} for abstract QDSs, we take advantage of fourth moment bounds.

Given~$f$ and~$\mu$, let
\beqn
c^{\ell,j}_n(k_1,\dots,k_\ell) = \mu(f_{n,k_1}\cdots f_{n,k_\ell}) - \mu(f_{n,k_1}\cdots f_{n,k_j}) \, \mu(f_{n,k_{j+1}}\cdots f_{n,k_\ell})
\eeqn
for all integers $2\le\ell \le 4$, $j\in\{1,\ell-1\}$ and $k_1,\dots,k_\ell\ge 0$. Note that if $c^{\ell,j}_n(k_1,\dots,k_\ell)$ is small, then the products $f_{n,k_1}\cdots f_{n,k_j}$ and $f_{n,k_{j+1}}\cdots f_{n,k_\ell}$ are nearly uncorrelated with respect to the initial distribution~$\mu$.  We also introduce the function
\beqn
\Phi(s) = 
\begin{cases}
s^{-1}(\log s)^{-2}, & s\ge 2
\\
2^{-1}(\log 2)^{-2}, & 0\le s< 2.
\end{cases}
\eeqn
The key property of $\Phi$ is its integrability.

\pagebreak
\begin{lem}\label{lem:QS_mean}
Let~$f:[0,1]\to\bR$ be a bounded measurable function and~$\mu$ a probability measure on~$[0,1]$. Suppose the following condition holds:
\smallskip
\begin{itemize}
\item[(A)] There exists $C>0$ such that
\beqn
|c^{\ell,j}_n(k_1,\dots,k_\ell)| \le C\Phi(k_{j+1}-k_j)
\eeqn
for all integers $2\le \ell \le 4$, $j\in\{1,\ell-1\}$ and $0\le k_1\le\dots\le k_\ell$. 
\end{itemize}
\medskip
Then
\beqn
\lim_{n\to\infty}\sup_{t\in [0,1]}|\bar\zeta_{n}(x,t)| = 0
\eeqn
for almost every~$x$ with respect to~$\mu$.
\end{lem}

\begin{proof}
By Lemma 5.1 in~\cite{Stenlund_2015}, Condition (A) implies
\beqn
\sum_{n=1}^\infty\mu(|\bar\zeta_n(t)|^4) < \infty
\eeqn
for all $t\in[0,1]$. In particular, given $t$, $\lim_{n\to\infty}\bar\zeta_n(x,t) = 0$, for almost every~$x$ with respect to $\mu$.  Moreover, the functions $t\mapsto\bar\zeta_n(x,t)$ are uniformly Lipschitz continuous. By Lemma~4.1 in~\cite{Stenlund_2015}, $\lim_{n\to\infty}\sup_{t\in [0,1]}|\bar\zeta_{n}(x,t)| = 0$.
\end{proof}

It remains to check that Condition (A) holds with $f\in C^1([0,1])$ and $\mu = m$. To that end, we set $f_0 = 1$ and $f_i = f$, $i\ge 1$, in Theorem~\ref{thm:multi}. This yields
\beqn
\begin{split}
|c^{\ell,j}_n(k_1,\dots,k_\ell)| 
& \le 4C_0 (2C_1)^{j+1}\|f\|_{C^1}^\ell \rho(k_{j+1}-k_j) 
\\
& \le 4C_0 \max(1,2C_1)^4\max(1,\|f\|_{C^1})^4 \rho(k_{j+1}-k_j).
\end{split}
\eeqn
Since $\rho$ is dominated by $\Phi$ under the standing assumption $\beta_*< \frac12$, we obtain~(A).

The result we have just proved also shows that the one-parameter family of measures $\mathscr{P} = (\hat\mu_{\gamma_t})_{t\in[0,1]}$ is a physical family of measures with a basin of full measure. Its uniqueness follows from the latter property of the basin and Corollary~2.5 in~\cite{Stenlund_2015}.

The proof of part (ii), and thus of Theorem~\ref{thm:main}, is now complete. \qed




\bigskip
\bigskip
\bibliography{Intermittent}{}
\bibliographystyle{plainurl}


\vspace*{\fill}

\end{document}